\documentclass[10pt,leqno]{amsart}
\usepackage{amsmath,amsfonts,amssymb,amsthm} 
\usepackage{epsfig}
\usepackage{graphicx}
\usepackage{psfrag} 
\newtheorem{theorem}{Theorem}[section]

\newtheorem{lemma}[theorem]{Lemma}

\theoremstyle{definition}

\newtheorem{remark}[theorem]{Remark}

\numberwithin{equation}{section}

\newcommand{\define}[1]{{\em #1\/}}

\newcommand{\cF}{{\mathcal F}}

\newcommand{\cH}{{\mathcal H}}

\newcommand{\field}[1]{\mathbb{#1}}
\newcommand{\N}{\field{N}}          		
\newcommand{\R}{\field{R}}          		
\newcommand{\Hy}{\field{H}}

\DeclareMathOperator{\supp}{spt}

\DeclareMathOperator{\dist}{dist}

\DeclareMathOperator\diver{div}

\DeclareMathOperator{\Vol}{Vol}

\DeclareMathOperator\spt{spt}

\DeclareMathOperator{\bd}{bd}

\DeclareMathOperator{\ir}{int}
\DeclareMathOperator\Int{Int}
\DeclareMathOperator\cone{Cone}

\begin{document}
\title[Asymptotic $H$-Plateau problem]{Asymptotic Plateau problem for prescribed mean curvature hypersurfaces}
\author[Jean-Baptiste Casteras]{Jean-Baptiste Casteras}
\address{D\'epartement de Math\'ematique, Universit\'e Libre de Bruxelles, CP 214, Boulevard du Triomphe, B-1050 Bruxelles, Belgium}
\email{jeanbaptiste.casteras@gmail.com}
\author[Ilkka Holopainen]{Ilkka Holopainen}
\address{Department of Mathematics and Statistics, P.O. Box 68, 00014 University of
Helsinki, Finland}
\email{ilkka.holopainen@helsinki.fi}
\author[Jaime B. Ripoll]{Jaime B. Ripoll}
\address{UFRGS, Instituto de Matem\'atica, Av. Bento Goncalves 9500, 91540-000 Porto
Alegre-RS, Brasil}
\email{jaime.ripoll@ufrgs.br}
\thanks{J.-B.C. supported by the FNRS project MIS F.4508.14; I.H. supported by the V\"ais\"al\"a Fund and the Magnus Ehrnrooth foundation; 
J.R. supported by the CNPq (Brazil) project 302955/2011-9}

\begin{abstract}
We prove the existence of solutions to the asymptotic Plateau problem for hypersurfaces of prescribed mean curvature in Cartan-Hadamard manifolds $N$. More precisely, given a suitable subset $L$ of the asymptotic boundary of $N$ and a suitable function $H$ on $N$, we are able to construct a set of locally finite perimeter whose boundary has generalized mean curvature $H$ provided that $N$ satisfies the so-called strict convexity condition and that its sectional curvatures are bounded from above by a negative constant. 
We also obtain a multiplicity result in low dimensions.
\end{abstract}

\maketitle
\section{Introduction}
In this note we consider an asymptotic Plateau problem for hypersurfaces of prescribed generalized mean curvature and asymptotic behaviour at infinity on Cartan-Hadamard manifolds that have strictly negative curvature upper bound $-\alpha^2$ and satisfy the so-called strict convexity condition. We call such hypersurfaces as \emph{$H$-mean curvature hypersurfaces}, where $H\colon N\to [-H_0,H_0]$ is a given continuous function, with $0\le H_0<(\dim N-1)\alpha$.
Recall that a Cartan-Hadamard manifold $N$ is a complete, connected, and simply connected Riemannian $n$-manifold, $n \geq 2$, of non-positive sectional curvature. It can be compactified by adding a sphere at infinity $\partial_\infty N$ and equipping the union $N\cup\partial_\infty N$ with the cone topology; see \cite{EO} for the details. We say that $N$ satisfies the \emph{strict
convexity condition (SC condition)} if, given $x\in\partial_{\infty}N$ and a
relatively open subset $W\subset\partial_{\infty}N$ containing $x,$ there
exists a $C^{2}$ open subset $\Omega\subset N$ such that
$x\in\ir \partial_{\infty}\Omega  \subset W$ and $N\setminus\Omega$ is convex. Here $\ir \partial_{\infty}\Omega$ denotes
the interior of $\partial_{\infty}\Omega$ in $\partial_{\infty}N$. 
For later purposes we also denote by $\bd$ the boundary relative to $\partial_\infty N$.
The SC condition was introduced by Ripoll and Telichevesky (\cite{RT}) in the context of an asymptotic Dirichlet problem for the minimal graph equation.

The $H$-mean curvature hypersurfaces are natural generalizations of minimal hypersurfaces. Anderson (\cite{AndInv}) solved the asymptotic Plateau problem in the hyperbolic space 
$\Hy^{n}$ for absolutely area minimizing varieties of arbitrary dimension and codimension. More precisely, he showed that, for any embedded closed submanifold $\Gamma^{p-1}\hookrightarrow \partial_\infty\Hy^{n}$ in the sphere at infinity, there exists a complete, absolutely area minimizing locally integral $p$-current $\Sigma \subset \Hy^{n}$ asymptotic to $\Gamma^{p-1}$ at infinity. By interior regularity results from geometric measure theory (see \cite{federer,morgan}), $\Sigma$ is smoothly embedded in $\Hy^{n}$ provided that $p=n-1\le 6$. Using similar methods based on mass minimizing currents, Bangert and Lang extended Anderson's results to Riemannian manifolds that are diffeomorphic and bi-Lipschitz equivalent to a 
Cartan-Hadamard manifold that has pinched sectional curvature $-b^2\le K\le -1$ (\cite{BL}) or that is Gromov hyperbolic and has bounded geometry (\cite{Lang}). In these results the topology of the minimal submanifold is not controlled. We refer to \cite{AndCMH,oliveirasoret,martinwhite,RTo} for existence results of complete minimal hypersurfaces with prescribed topology.

Our research is inspired by the papers of Alencar and Rosenberg \cite{Al-Ro} and of Tonegawa \cite{Tone} who studied the asymptotic Plateau problem for constant mean curvature hypersurfaces in the hyperbolic space. Tonegawa also investigated the regularity up to the boundary. For existence results of constant mean curvature hypersurfaces in $\Hy^n$ that can be represented as graphs we refer to \cite{guanspruck,nellispruck,RT_BBMS}. We refer to \cite{cosku,coskuTAMS} for a survey and recent developments on the $H$-Plateau problem in the hyperbolic space. 
In this paper we obtain the hypersurfaces $M$ of generalized mean curvature $H$ as reduced boundaries of sets of locally finite perimeter $Q$ in $N$ that locally minimizes a family of functionals
\begin{equation}\label{glob-funct}
\int_{Q\cap C}H(x) dm(x) + \int_{\partial^* Q\cap C}d\cH^{n-1} 
\end{equation}
for all compact sets $C\subset N$.

In \cite{CHR3} we searched geometric conditions on $N$ that would imply the SC condition.
We were able to verify the SC condition on manifolds whose curvature lower bound could go to $-\infty$ and upper bound to $0$ simultaneously at certain rates, or on some manifolds whose sectional curvatures go to $-\infty$ faster than any prescribed rate. For instance, the SC condition holds on $N$ if 
\[
 -\dfrac{e^{2k r(x)}}{r(x)^{2+2\varepsilon}}\leq K_N (x) \leq -k^2,
 \]
 or 
 \[
 -c\,e^{(2-\varepsilon)r(x)}e^{\displaystyle{e^{r(x)/e^3}}} \le K_N(x) \le -\phi e^{2r(x)}
 \]
for all $x\in N$ such that $r(x)=d(x,o)\geq R^\ast $, $R^\ast$ large enough, and where 
$k>0,\ \varepsilon>0,\ \phi>1/4$, and $c>0$ are constants. We refer to 
\cite[Theorem 1.3]{CHR3} for more precise conditions. The SC condition was also applied in \cite{CHR3} to solve  asymptotic Dirichlet and Plateau problems on Cartan-Hadamard manifolds.

Our main results read as follows. 
\begin{theorem}\label{main1}
Let $N^n$ be a Cartan-Hadamard manifold satisfying the SC condition. Suppose also that sectional curvatures of $N$ have a negative upper bound
\[
K_N\le-\alpha^2<0,\quad\alpha>0.
\]
Let $L\subset\partial_\infty N$ be a relatively closed subset such that 
$L=\bd U=\bd U'$ for some disjoint relatively open subsets $U,U'\subset\partial_\infty N$ with $U\cup U'=\partial_\infty N\setminus L$. Suppose that $H\colon N\to[-H_0,H_0]$ is a continuous function where $0\le H_0<(n-1)\alpha$. Then there exists an open set $Q\subset N$ of locally finite perimeter whose boundary $M$ has (generalized) mean curvature $H$ towards $N\setminus Q,\ \partial_\infty Q\subset U\cup L$, and $\partial_\infty M=L$. 
\end{theorem}

Following \cite{Al-Ro}, we also obtain a multiplicity result: 
\begin{theorem}\label{main2}
Let $N^n$  and $L\subset\partial_\infty N$ be as above. If $H\in ]0,H_0]$ is a constant, there exist two disjoint open sets $Q_j\subset N, j=1,2$, of locally finite perimeter whose boundaries $M_j$ have constant (generalized) mean curvature $H$ and $\partial_\infty M_j=L$. Moreover, if $n\le 7$, then  $M_1$ and $M_2$ are disjoint.
\end{theorem}

It seems that some kind of convexity condition is necessary to solve the asymptotic Plateau problem. Indeed, Kloeckner and Mazzeo \cite[Corollary 5.2]{KlMa} proved, in the case $N^3=\mathbb{H}^2\times \R$, which clearly does not satisfy the SC condition, that the asymptotic Plateau problem for the geodesic compactification (the one described above), is solvable only for very specific families of curves $L$. On a related topic, Coskunuzer \cite{coskupreprint} also obtained non-existence results in the case $N^3=\mathbb{H}^2\times \R$ but with the product compactification i.e. $\partial_\infty N= (\partial_\infty \mathbb{H}^2 \times \R) \cup (\partial_\infty \mathbb{H}^2 \times \partial_\infty \R) \times (\mathbb{H}^2 \times \partial_\infty \R )$. A similar result has been obtained in \cite{KlMeRa} when $N$ is the homogenuous space $\mathbb{E} (-1 , \tau)$. Notice that in all the previous examples, the sectional curvatures are null in some directions. The asymptotic Plateau problem for minimal hypersurfaces might be solvable on manifolds satisfying, instead of the SC condition, a \emph{strict mean convexity condition}, where the convexity of sets $N\setminus\Omega$ is replaced by mean convexity. We also believe that it should be possible to allow the sectional curvatures of $N$ to go to $0$ not too fast (as in our results concerning the asymptotic Plateau problem for minimal submanifolds \cite{CHR3}). But currently, we do not know how to prove Lemma~\ref{traplemma} in such settings; see in particular \eqref{globHcond}.

The plan of this note is the following: in Section~\ref{gmc} we introduce the variational problem we are going to consider and prove the first variation formula for smooth hypersurfaces. Using the first variation formula we then obtain a maximum principle. Then, in Section~\ref{app} we apply the SC condition and the curvature upper bound to prove a ``trapping" lemma (Lemma~\ref{traplemma}). Then we consider the asymptotic $H$-Plateau problem and prove our main results.

\subsection*{Acknowledgement}
We wish to thank the referee for suggestions, which led to improvements in the presentation.

\section{Prescribed (generalized) mean curvature}\label{gmc}
Let $N^n$ be an $n$-dimensional Cartan-Hadamard manifold. 
Suppose that $B=B(o,R)$ is a geodesic (open) ball of radius $R$ centered at $o\in N$ and that $\Sigma\subset \partial B$ is a 
relatively open set, with $\cH^{n-2}$-rectifiable boundary $\Gamma=\partial\Sigma$ (relative to $\partial B$). 
Consider open sets $Q\subset N$ of finite perimeter and with (topological) boundary 
\[
\partial Q=M\cup \Gamma\cup \Sigma,
\]
where $M\subset N$ is an $\cH^{n-1}$-rectifiable topological, not necessarily connected, $(n-1)$-manifold.  

Given a continuous function $H\colon N\to\R$ we consider the following (Massari) functional $\cF$,
\begin{equation}\label{massari}
\cF(M)= \int_{Q}H(x)dm(x) + \int_{M}d\cH^{n-1},
\end{equation} 
where $M$ and $Q$ are as above. We call such $M$ admissible for the functional $\cF$ with boundary data $\Gamma$.

We call $H$ the \define{prescribed mean curvature function}. The terminology is justified by the fact that a stationary $M$ of the functional \eqref{massari} has mean curvature $H(x)$ pointing outwards $Q$ at every point $x\in M$ where $M$ is smooth; see e.g. \cite{massari} and \cite[17.3]{maggi}. We refer to \cite[Theorem 17.5 and Prop. 17.8]{maggi} for first variation formulae of perimeter and potential energy.
For our purposes, the first variation formula \eqref{eq1var} for a \define{smooth} hypersurface $M$ is sufficient. For the convenience of reader's we sketch its proof. To this end, suppose that $M$ is smooth and admissible for \eqref{massari} with boundary data $\Gamma$.
Let $X$ be a smooth  
vector field such that $X=0$ on $\Gamma$ and 
let $\Phi$ be the flow of $X$.  
For sufficiently small $|t|$ we denote
$M_t=\{\Phi(x,t)\colon x\in M\}$ and by $Q_t$ the set bounded by 
$M_t\cup \Gamma \cup \Sigma$. To obtain the first variation formula for smooth $M$ we compute
\begin{align*}
\delta\cF(M)[X]&:=\frac{d}{dt}\cF(M_t)_{|t=0}\\
&=
\frac{d}{dt}\left(\int_{Q_t}H(x)dm(x)+\int_{M_t}d\cH^{n-1}\right)_{|t=0}\\
&=\frac{d}{dt}\left(\int_{Q_t}H(x)dm(x)\right)_{|t=0}
+\int_M \diver_M X d\cH^{n-1}\\
&=\frac{d}{dt}\left(\int_{Q_t}H(x)dm(x)\right)_{|t=0}
-\int_M \langle X,\vec{H}\rangle d\cH^{n-1},
\end{align*}
where $\vec{H}$ is the mean curvature vector field of $M$.
To compute the first term on the right-hand side 
we apply the co-area formula. For that purpose we define in 
$\supp X$ and for sufficiently small $|t|$  
a smooth function $f$ by setting 
$f(x)=t$ for $x\in M_t$. Thus $f=0$ and $\nabla f=X^\perp/|X^\perp|^{2}$ on 
$M=M_0$. 
Hence
\begin{align*}
&\lim_{t\to 0}\frac{1}{t}\left(
\int_{Q_t}H(x) dm(x)-\int_{Q}H(x)dm(x)\right)\\
=&\lim_{t\to 0}\frac{1}{t}\bigg(
\int_{Q_t\setminus Q}|\nabla f(x)||\nabla f(x)|^{-1}H(x) dm(x)\\
&\quad-\int_{Q\setminus Q_t}|\nabla f(x)||\nabla f(x)|^{-1}H(x)dm(x)\bigg)\\
=&\int_{M}\langle X_x,\nu_x\rangle H(x) d\cH^{n-1}(x),
\end{align*}
where $\nu_x$ is the unit normal to $M$ at $x$ pointing outwards $Q$. 
To verify the last step above, we may assume $t\to 0^+$. Then the set 
$Q_t\setminus Q$, whenever nonempty for small $t>0$, corresponds those $x\in M$ where $X_x$ points outwards $Q$, that is 
$\langle X_x,\nu_x\rangle >0$. The co-area formula now implies
\begin{align*}
&\lim_{t\to 0+}\frac{1}{t}\int_{Q_t\setminus Q}|\nabla f(x)||\nabla f(x)|^{-1}H(x) dm(x)\\
=&\lim_{t\to 0+}
\frac{1}{t}\int_0^t\left(\int_{M_s\cap(Q_t\setminus Q)}|\nabla f(x)|^{-1}H(x) d\cH^{n-1}(x)\right)ds\\
=&\int_{M\cap\{\langle X,\nu\rangle >0\}}\langle X_x,\nu_x\rangle H(x)d\cH^{n-1}(x).
\end{align*}
Similarly,
\begin{align*}
\lim_{t\to 0+}\frac{1}{t}\int_{Q\setminus Q_t}&|\nabla f(x)||\nabla f(x)|^{-1}H(x) dm(x)\\
&=-\int_{M\cap\{\langle X,\nu\rangle <0\}}\langle X_x,\nu_x\rangle H(x)d\cH^{n-1}(x).
\end{align*}
Thus we obtain the first variation formula 
\begin{equation}\label{eq1var}
\delta\cF(M)[X]=
\int_M \big\langle X_x,H(x)\nu_x -\vec{H}_x\big\rangle d\cH^{n-1}(x).
\end{equation}

Assuming that a smooth hypersurface $M$ is stationary, i.e.
\[
\delta\cF(M)[X]=0
\]
for all $X$ as above, we conclude that 
 $M$ has prescribed  mean curvature $H(x)$ pointing outwards $Q$.
 
 Fix $o\in N$ and denote by $H(x,r)$ the scalar (inwards pointing) mean curvature of the geodesic sphere 
$S(o,r)=\partial B(o,r)$ at $x\in S(o,r)$. Hence $H(x,r)=\Delta r(x)$, where $r(x)=d(x,o)$, the Riemannian distance between $o$ and $x$.
Suppose that the prescribed mean curvature function 
$H$ satisfies 
\begin{equation}\label{Hcond}
|H(x)|< H(x,r)=\Delta r(x)
\end{equation}
for all $x\in S(o,r)$ and $r>0$.

We have the following maximum principle which is the counterpart of \cite[Lemma 5]{gulliver}, where the proof is given in dimension three in a slightly different setting.
\begin{lemma}\label{lemma5}
Let $M\subset N$ be admissible for $\cF$ with boundary data $\Gamma\subset\partial B,\ B=B(o,R),$ such that $M\setminus\overline{B}\ne\emptyset$. Then there exists $M'\subset\overline{B}$ admissible for $\cF$ with boundary data $\Gamma$ such that $\cF(M')<\cF(M)$. In particular, if $M_0$ is a minimizer of the functional \eqref{massari} with boundary data $\Gamma$, then 
$M_0\subset\overline{B}$.
\end{lemma}

\begin{proof}
Let $Q$ be the open set corresponding the set $M$ and let 
$U=Q\setminus\overline{B}$. Note that $U$ is a nonempty open set, and therefore it has positive volume. 
By \cite[Theorem 3.42]{AmbFusPal} there are open sets $Q_i$ with smooth boundaries such that $Q_i$ converges to $Q$ in measure and $\cH^{n-1}(\partial Q_i)=P(Q_i)\to P(Q)$ as 
$i\to \infty$. Here $P(Q)$ stands for the perimeter of $Q$. Hence we may suppose that there exists a sequence of smooth $(n-1)$-manifolds $M_i$ admissible for $\cF$ with boundary data $\Gamma$ such that 
$M_i\setminus\overline{B}\neq\emptyset,\ M_i\subset B(o,R+s)$ for some $s>0$, 
and $\cF(M_i)\to\cF(M)$. 
Therefore it is enough to show that for each $M_i$ there exists 
$M'_i\subset\overline{B}$ admissible for $\cF$ with boundary data 
$\Gamma$ such that $\cF(M'_i)\le \cF(M_i)-c$ for some positive constant $c$ independent of $i$.
More precisely, we prove that 
\begin{equation}\label{decay}
\cF(M'_i)\le\cF(M_i)-\frac{1}{2}\int_{U_i}\big((1-\varepsilon)\Delta r(x)-|H(x)|\big)dm(x),
\end{equation}
where $U_i=Q_i\setminus\overline{B}$ and $\varepsilon>0$ is such that
$|H(x)|<(1-\varepsilon)\Delta r(x)$ for all $x\in \overline{B}(o,R+s)$. 
Since $Q_i$ converges to $Q$ in measure, 
we obtain an admissible $M'\subset\overline{B}$ as the Hausdorff limit of $M'_i$'s, with
 \begin{equation}\label{decay2}
\cF(M')\le\cF(M)-\frac{1}{4}\int_{U}\big((1-\varepsilon)\Delta r(x)-|H(x)|\big)dm(x)
<\cF(M).
\end{equation}
The integral in \eqref{decay2} is positive since $U$ has positive volume and
$(1-\varepsilon)\Delta r -|H|$ is positive in $U$.

To prove \eqref{decay} we fix $M_i$ and let $x_1$ be a point of $M_i$ at the maximal distance from $o$. 
Then the mean curvature vector $\vec{H}_{x_1}$ of $M_i$ at $x_1$ is parallel to 
$-\nabla r(x_1)$ and  
$|\vec{H}_{x_1}|\ge \Delta r(x_1)$. Let $r_1$ be the infimum of all $r\ge R$ such that 
\begin{equation}\label{Hcond2}
\langle\vec{H}_x,\nabla r(x)\rangle < 0\quad\text{and}\quad 
|\vec{H}_x|\ge (1-\varepsilon)\Delta r(x)> |H(x)|
\end{equation}
 for all $x\in M_i$ with $r(x)\ge r$. We notice that $r_1<r(x_1)$.  
Next we define a complete smooth vector field $X_1$ on $N$ by setting
$X_1(x)=-\rho_1\big(r(x)\big)\nabla r(x)$, where  
$\rho_1\colon [0,\infty)\to [0,\infty)$ is a smooth 
function such that $\rho_1(s)>0$ if and only if $s>r_1$.
Let $\Phi_1\colon\R\times N\to N$ be the flow of $X_1$ and denote
\[
M_i^t=\{\Phi_1(t,x)\colon x\in M_i\}
\]
for $t>0$. Furthermore, let 
\begin{align*}
U_{i,1}&=\{\Phi_1(t,x)\colon t>0,x\in M_i\setminus\overline{B}(o,r_1)\}\\
&=\left\lbrace \exp_o tv\colon r_1/|v|<t<1,\ v\in\exp_o^{-1}\big(M_i\setminus\overline{B}(o,r_1)\big)\right\rbrace
\end{align*}
be the open set whose boundary is contained in 
$\big(M_i\setminus\overline{B}(o,r_1)\big)\cup\partial B(o,r_1)$. Notice that
$U_{i,1}\subset U_i$ by \eqref{Hcond2}, in fact, $U_{i,1}=U_i\setminus\overline{B}(o,r_1)$. Finally, let 
$\widetilde{M}_{i,1}$ be the Hausdorff limit of $M_i^t$ as $t\to\infty$. Thus 
\[
\widetilde{M}_{i,1}= 
\big(M_i\cap\overline{B}(o,r_1)\big)\cup 
\left\lbrace \exp_o (r_1 v/|v|)\colon v\in\exp_o^{-1}\big(M_i\setminus\overline{B}(o,r_1)\big)\right\rbrace,
\]
the union of $M_i\cap\overline{B}(o,r_1)$ 
 and the radial projection of $M_i\setminus\overline{B}(o,r_1)$ into $\partial B(o,r_1)$. Notice that 
 $M_i\cap B(o,r_1)=\widetilde{M}_{i,1}$.
Using \eqref{Hcond2} and the co-area formula we obtain
\begin{align}\label{co-area}
\cF(\widetilde{M}_{i,1})&=\cF(M_i)+\int_{0}^{\infty}\delta\cF(M_i^t)[X_1]dt\nonumber \\
&=\cF(M_i)-\int_{0}^{\infty}\left(\int_{M_{i,1}^t}\langle X_1(x),\vec{H}_x -H(x)\nu_x\rangle d\cH^{n-1}(x)\right)dt \\
&\le \cF(M_i)-\int_{U_{i,1}}\big((1-\varepsilon)\Delta r(x)-|H(x)|\big)dm(x).\nonumber
\end{align}
To verify the inequality in \eqref{co-area} we define a smooth function $f$ on $U_{i,1}$ by setting $f=t$ on $M_{i,1}^t\cap U_{i,1}$ and apply the co-area formula with the function $f$. Now $\nabla f=X_1^\perp/|X_1^\perp|^2$ on $M_{i,1}^t\cap U_{i,1}$, where $X_1^\perp$ is the orthogonal component of $X_1$ with respect to $M_{i,1}^t$. Thus
\[
|\nabla f(x)|^{-1}=|X_1^\perp (x)|=\langle X_1(x),\vec{H}_x/|\vec{H}_x|\rangle
\]
for $x\in M_{i,1}^t$, where $\vec{H}$ is the mean curvature vector field of 
$M_{i,1}^t$. 
Since $|\vec{H}_x|>|H(x)|$ and $\vec{H}_x$ is parallel to $-\nu_x$, the vector $\vec{H}_x$ is parallel to $\vec{H}_x-H(x)\nu_x$, and therefore
\[
\langle X_1(x),\vec{H}_x -H(x)\nu_x\rangle=-\rho_1(x)\langle \nabla r(x),
\vec{H}_x -H(x)\nu_x\rangle>0
\]
in $U_{i,1}$. More precisely,
\begin{align*}
\langle X_1(x),\vec{H}_x -H(x)\nu_x\rangle
&\ge (|\vec{H}_x|-|H(x)|)\langle X_1,\vec{H}_x/|\vec{H}_x|\rangle\\
&\ge \big((1-\varepsilon)\Delta r(x)-|H(x)|\big)|\nabla f(x)|^{-1}
\end{align*}
in $U_{i,1}$. We obtain
\begin{align*}
\int_{U_{i,1}}&\big((1-\varepsilon)\Delta r(x)-|H(x)|\big)dm(x)\\
&=
\int_{U_{i,1}}\big((1-\varepsilon)\Delta r(x)-|H(x)|\big)|\nabla f(x)|^{-1}|\nabla f(x)| dm(x)x\\
&=\int_0^\infty\left(\int_{M_{i,1}^t}\big((1-\varepsilon)\Delta r(x)-|H(x)|\big)|\nabla f(x)|^{-1}d\cH^{n-1}(x)\right)dt\\
&\le
\int_{0}^{\infty}\left(\int_{M_{i,1}^t}\langle X_1(x),\vec{H}_x -H(x)\nu_x\rangle d\cH^{n-1}(x)\right)dt,
\end{align*}
and the estimate \eqref{co-area} follows.

If $r_1=R$, we obtain \eqref{decay} since $\widetilde{M}_{i,1}\subset\overline{B}$ and $U_{i,1}=U_i$ in that case. 
If $r_1>R$ we continue by smoothing out $\widetilde{M}_{i,1}$ in a neighborhood of $M_i\cap\partial B(o,r_1)$  to obtain a smooth 
$M_{i,1}\subset U_i\cap\overline{B}(o,r_1)$ admissible for $\cF$ with boundary data $\Gamma$. We denote 
by $Q_{i,1}$ the open set bounded by $M_{i,1}\cap\Gamma\cap\Sigma$. We also write $\widetilde{Q}_{i,1}=Q_i\cap B(o,r_1)$ and
$A_{i,1}=\widetilde{Q}_{i,1}\setminus Q_{i,1}$.
The smoothing can be done such that the volume of $A_{i,1}$ and 
$\cH^{n-1}(M_{i,1}\setminus\widetilde{M}_{i,1})$ are as small as we wish, in particular, so that
\begin{align*}
\cF(M_{i,1})-\cF(\widetilde{M}_{i,1})&\le -\int_{A_{i,1}}H(x)dm(x) +\cH^{n-1}\big(M_{i,1}\setminus \widetilde{M}_{i,1}\big)\\
&<\frac{1}{4}\int_{U_i}\big((1-\varepsilon)\Delta r(x)-|H(x)|\big) dm(x).
\end{align*}
We repeat the argument above by choosing a point $x_2\subset M_{i,1}$ at the maximal distance from $o$ and letting $r_2$ be the infimum of $r\ge R$ such that \eqref{Hcond2} holds for all $x\in M_{i,1}$, with $r(x)\ge r$. Then $r_2<r_1$ by the smoothing process. We obtain an admissible $\widetilde{M}_{i,2}\subset\overline{B}(o,r_2)$ with 
\[
\cF(\widetilde{M}_{i,2})\le \cF(M_{i,1})-\int_{U_{i,2}}\big((1-\varepsilon)\Delta r(x)-|H(x)|\big)dm(x),
\]
where
\[
U_{i,2}=\left\lbrace \exp_o tv\colon r_1/|v|<t<1,\ v\in\exp_o^{-1}\big(M_{i,1}\setminus\overline{B}(o,r_2)\big)\right\rbrace\subset U_i.
\]
If $r_2=R$, we are done. Otherwise we smooth out 
$\widetilde{M}_{i,2}\subset\overline{B}(o,r_2)$ to obtain an admissible smooth 
$M_{i,2}\subset\overline{B}(o,r_2)$ with
\[
\cF(M_{i,2})-\cF(\widetilde{M}_{i,2})<\frac{1}{4^2}\int_{U_i}\big((1-\varepsilon)\Delta r(x)-|H(x)|\big) dm(x).
\] 
By continuing this way we get a strictly decreasing (possible finite) sequence 
$r_j\ge R$ and admissible $\widetilde{M}_{i,j}, M_{i,j}\subset\overline{B}(o,r_j)$,
where $M_{i,j}$ is smooth, such that
\begin{align*}
\cF(M_{i,j})<\cF(M_{i,j-1})&-\int_{U_{i,j}}\big((1-\varepsilon)\Delta r(x)-|H(x)|\big)dm(x)\\
&+\frac{1}{4^j}\int_{U_i}\big((1-\varepsilon)\Delta r(x)-|H(x)|\big) dm(x).
\end{align*}
At each step $r_j$ is the infimum of $r\ge R$ such that  \eqref{Hcond2} holds for all $x\in M_{i,j-1}$, with $r(x)\ge r$.
Hence
\begin{align*}
\cF(M_{i,j})<\cF(M_i)&-\int\limits_{\bigcup_{k=1}^j U_{i,k}} \big((1-\varepsilon)\Delta r(x)-|H(x)|\big) dm(x)\\
&+\sum_{k=1}^j \left(\frac{1}{4}\right)^k \int_{U_i}\big((1-\varepsilon)\Delta r(x)-|H(x)|\big) dm(x)\\
<\cF(M_i)&-\int\limits_{\bigcup_{k=1}^j U_{i,k}} \big((1-\varepsilon)\Delta r(x)-|H(x)|\big) dm(x)\\
&+\frac{1}{3} \int_{U_i}\big((1-\varepsilon)\Delta r(x)-|H(x)|\big) dm(x).
\end{align*}
If $r_k=R$ for some $k$, we are done. Similarly, if $r_\infty:=\lim_j r_j =R$, the Hausdorff limit $M'_i$ of $M_{i,j}$'s is admissible, 
$M'_i\subset\overline{B}$, and satisfies
\begin{align}\label{mprimedecay}
\cF(M'_i)<\cF(M_i)&-\int_{\cup_j U_{i,j}}\big((1-\varepsilon)\Delta r(x)-|H(x)|\big)dm(x)\\
&+\frac{1}{3} \int_{U_i}\big((1-\varepsilon)\Delta r(x)-|H(x)|\big) dm(x).\nonumber
\end{align}
On the other hand, if $r_\infty:=\inf r_j>R$, the Hausdorff limit 
$M'_i\subset\overline{B}(o,r_\infty),\ M'_i\cap\partial B(o,r_\infty)\neq\emptyset$, and \eqref{mprimedecay} holds. 
Repeating the process once more gives a contradiction with the definition of 
$r_\infty$. Moreover, the smoothing processes can be done so that the set 
$U_i\setminus\big(\cup_j U_{i,j}\big)$ has as small volume as we wish; see \cite{schmidt}. Finally, since the sets $U_i$ converge to $U$ in measure, we obtain $M'\subset\overline{B}$ admissible for $\cF$ with boundary data $\Gamma$ such that
\[
\cF(M')\le\cF(M)-\frac{1}{4}\int_U \big(\Delta r(x)-|H(x)|\big)dm(x).
\]  
\end{proof}
\begin{remark}
\begin{enumerate}
\item In fact, the constant $1/4$ in \eqref{decay2} can be replaced by 1. Indeed, instead of the constant ratio $1/4$ in geometric series we may use a constant $q>0$ as small as we wish and perform the smoothing processes so that the lost of volume is as small as we like. Then $1/4$ in \eqref{decay2} can be replaced by any constant $c<1$, hence also by 1.
\item Consider open sets $Q\subset N$ of finite perimeter with topological boundary of the form $\partial Q=M\cup\Gamma\cup\Sigma$, where $\Sigma$ is a relatively compact $\cH^{n-1}$-rectifiable topological $(n-1)$-manifold whose relative boundary
$\Gamma=\partial\Sigma$ is $\cH^{n-2}$-rectifiable. We define the (Massari) functional $\cF$ as in $\eqref{massari}$ also in this setting. Suppose that 
$\Gamma\subset\overline{B}(o,R)$ for some $R>0$. The proof of Lemma~\ref{lemma5} then gives that any minimizer of $\cF$ with boundary data $\Gamma$ must stay inside $\overline{B}(o,R)$.
\item The existence of minimizers is proven below right before Lemma~\ref{traplemma}.
\end{enumerate}
\end{remark}

\section{Asymptotic $H$-Plateau problem}\label{app}
Throughout this section we assume that $N$ satisfies the SC condition and that
the sectional curvatures of $N$ are bounded from above by a negative constant $-\alpha^2$. Suppose, furthermore, that 
\[
|H(x)|\le H_0<(n-1)\alpha.
\]
Then there exists 
$r_0$ such that $(n-1)\alpha\tanh(\alpha r_0)>H_0$.
If $A\subset N$ is an open convex set, we denote by $A_0$ the (open)
$r_0$-neighborhood of $A$ and we write $A_0^+=N\setminus\overline{A_0}$. 

Let $U$ and $U^\prime$ be disjoint relatively open sets in $\partial_\infty N$ such that $U\cup U^\prime=\partial_\infty N\setminus L$, where $L=\bd U = \bd U^\prime$. 
Furthermore, let $O$  be the (open) tubular neighborhood of $\cone_o L$ 
with fixed positive radius $\delta$ and let $\widetilde{O}$ be the tubular neighborhood of $\cone_o U$ with the same radius $\delta$.
Note that $O\subset\widetilde{O}$.
As in \cite[Lemma 4.1]{CHR3} we take, for each $x\in\partial_\infty N\setminus L$, a relatively open set 
$W_x\subset\partial_\infty N\setminus L$ such that $x\in W_x$ and then apply the SC condition and the Hessian comparison to obtain a $C^2$ smooth open subset 
$\Omega_x$ of $N$ such that $x\in\ir\partial_\infty \Omega_x
\subset W_x$, $N\setminus\Omega_x$ is convex, and $\Omega_x\cap O=\emptyset$ for every $x\in\partial_\infty N\setminus L$.

We define
\begin{align*}
\Lambda&=\bigcap_{x\in\partial_\infty N\setminus L}(N\setminus\Omega_x),\\
\widetilde{\Lambda}&=\bigcap_{x\in U^\prime}(N\setminus\Omega_x).
\end{align*}
Then  $O\subset\Lambda, \widetilde{O}\subset\widetilde{\Lambda}$, $\Lambda$ and $\widetilde{\Lambda}$ are closed in $N$, $\Int\Lambda$ and $\Int\widetilde{\Lambda}$ are open convex subsets of $N$ such that 
$\partial_\infty\Lambda= L$ and $\partial_\infty\widetilde{\Lambda}=L\cup U$.

For a fixed $R>0$ 
we approximate $\partial B(o,R)\cap\widetilde{O}$ by a relatively 
open subset $\Sigma\subset\partial B(o,R)$ so that its boundary 
$\Gamma:=\partial\Sigma$ is an $\cH^{n-2}$-rectifiable 
subset of $\partial B(o,R)\cap O$. We consider 
$[\Sigma]$ as a rectifiable $(n-1)$-current and 
$[\Gamma]=\partial [\Sigma]$ as a rectifiable $(n-2)$-current.
Let $M_{j},\ j\in\N,$ be a minimizing sequence for the functional 
\eqref{massari} with boundary data $\Gamma$. Then there exists a subsequence, still denoted by $M_{j}$, and an $\cH^{n-1}$-rectifiable set $M$ such that $[M_{j}]\to [M]$ and 
$\partial [M] =[\Gamma]$ as currents. 
It follows that $M$ is a minimizer for \eqref{massari} with boundary data $\Gamma$, hence $M\subset\overline B(o,R)$ by 
Lemma~\ref{lemma5}. 
Next we prove that $M\subset \overline{B}(o,R)\cap\Lambda_0$, where $\Lambda_0$ is the $r_0$-neighborhood of $\Lambda$, cf. \cite[Theorem 1.1]{coskuGD}.
  
  \begin{lemma}\label{traplemma}
Let $\Gamma\subset\partial B(o,R)\cap O$ be the boundary of $\Sigma\subset\partial B(o,R)\cap\widetilde{O}$ as above. Then there exists a minimizer $M$ of \eqref{massari}
with the boundary data $\Gamma$ such that $M\subset \overline{B}(o,R)\cap\overline{\Lambda}_0$. 
\end{lemma}

\begin{proof}
The existence of a minimizer $M$ was obtained above.
By Lemma \ref{lemma5}, it suffices to prove that $M\subset\Lambda_0$. Hence it is enough to show that, for all 
$x\in \partial_\infty N\setminus L$, $M\subset \overline{(N\setminus\Omega_x)}_0$, where $(N\setminus\Omega_x)_0$ is the 
$r_0$-neighborhood of the ($C^2$-smooth) convex set $N\setminus \Omega_x$.  
Let $d=\dist(\cdot,N\setminus\Omega_x)$ be the Riemannian distance to the convex set $N\setminus\Omega_x$. 
Then $d$ is $C^2$-smooth in $N\setminus\Omega_x$ and the Hessian comparison theorem \cite{Kasue} implies that
\begin{equation}\label{globHcond}
\Delta d(z)\ge (n-1)\alpha\tanh\big(\alpha d(z)\big)
\ge (n-1)\alpha\tanh\big(\alpha r_0\big)>H_0\ge |H(z)|
\end{equation}
for all $z\in (N\setminus\Omega_x)_0^+=N\setminus\overline{(N\setminus\Omega_x)_0}$. 
Assume on the contrary that $M\cap (N\setminus\Omega_x)_0^+\ne\emptyset$. 
As in the proof Lemma~\ref{lemma5} there is a sequence 
of smooth $(n-1)$-manifolds $M_i$ admissible for the functional 
$\cF$ with boundary data $\Gamma$ such that $M_i\cap (N\setminus\Omega_x)_0^+ \ne\emptyset$ and $\cF(M_i)\to\cF(M)$. Again, for each fixed $M_i$, if $x_1\in M_i$ is at maximal distance from $N\setminus\Omega_x$, the mean curvature vector $\vec{H}_{x_1}$ of $M_i$ is parallel to $-\nabla d(x_1)$ and 
$|\vec{H}_{x_1}|\ge\Delta d(x_1)>H(x_1)$.
Hence
$\langle\vec{H}_x,\nabla d(x)\rangle <0$ and 
$|\vec{H}_x|> |H(x)|$
 for all $x\in M_i$ with $d(x)\ge d_1=d(x_1)-\varepsilon$. 
Next we define a smooth vector field on $N$ by setting
$X_x=-\rho\big(d(x)\big)\nabla d(x)$, where  
$\rho\colon [0,\infty)\to [0,\infty)$ is a smooth non-decreasing function such that $\rho(t)>0$ if and only if $t>d_1$. 
Then 
$\langle X_x,H(x)\nu_x-\vec{H}_x\rangle <0$ for $d(x)>d_1$ and vanishes elsewhere.
Applying the first variation formula to $M_i$ with the vector field 
$X$ we obtain
\[
\delta\cF(M_i)[X]=
\int_{M_i} \big\langle X_x,H(x)\nu_x-\vec{H}_x\big\rangle d\cH^{n-1}(x)<0.
\]
In particular, deforming $M_i$ in the direction of $X$ decreases 
the value of the functional and repeating a similar argument as in the proof of Lemma~\ref{lemma5} we may conclude that the minimizer $M$ must stay inside $(N\setminus\Omega_x)_0$.
\end{proof}
Now we are ready to prove our first main theorem. For convenience we repeat it from Introduction.
\begin{theorem}\label{main1b}
Let $N^n$ be a Cartan-Hadamard manifold satisfying the SC condition. Suppose also that sectional curvatures of $N$ have a negative upper bound
\[
K_N\le-\alpha^2<0,\quad\alpha>0.
\]
Let $L\subset\partial_\infty N$ be a relatively closed subset such that 
$L=\bd U=\bd U'$ for some disjoint relatively open subsets $U,U'\subset\partial_\infty N$ with $U\cup U'=\partial_\infty N\setminus L$. Suppose that $H\colon N\to[-H_0,H_0]$ is a continuous function where $0\le H_0<(n-1)\alpha$. Then there exists an open set $Q\subset N$ of locally finite perimeter whose boundary $M$ has (generalized) mean curvature $H$ towards $N\setminus Q,\ \partial_\infty Q\subset U\cup L$, and $\partial_\infty M=L$. 
 \end{theorem}
\begin{proof}
The proof is an adaptation of a method of Lang \cite{Lang1}; see also \cite{BL} and \cite{CHR3}.
We apply Lemma~\ref{traplemma} in geodesic balls $B(o,R_i)$, where 
$R_i\to\infty$ is an increasing sequence. 
For each $i$ 
we approximate $\partial B(o,R_i)\cap\widetilde{O}$ by a relatively 
open subset $\Sigma_i\subset\partial B(o,R_i)$ so that $\Gamma_i:=\partial\Sigma_i\subset
\partial B(o,R_i)\cap O$ is $\cH^{n-2}$-rectifiable. Furthermore, let 
$M_i\subset\overline{B}(o,R_i)\cap \Lambda_0$ be a minimizer of \eqref{massari} with the boundary data $\Gamma_i$ given by Lemma~\ref{traplemma}. Then $\overline{M}_i\cup\Sigma_i$ bounds an open set $Q_i$ of finite perimeter. Furthermore, 
total variations are locally uniformly bounded as
\[
\sup_i ||\partial[Q_i]||\big(B(o,r)\big)\le \cH^{n-1}\big(\partial B(o,r)\big)+ H_0\Vol\big(B(o,r)\big), 
\]
the upper bound depending only on $r,\ n$, and the sectional curvature lower bound on $B(o,r)$. 
Hence there exists a subsequence $Q_{i_j}$ and a set $Q\subset N$ of locally finite perimeter with boundary $M=\partial Q$ such that $[Q_{i_j}]\to [Q]$ and $[M_{i_j}]\to [M]$ as currents and that $M$ is a minimizer for $\cF$ in the sense that $\delta\cF(M)[X]=0$
for all compactly supported smooth vector fields. Hence $M$ has (generalized) mean curvature $H$ pointing outwards $Q$. Furthermore, since $M\subset\Lambda_0$ and $\partial_\infty\Lambda_0= L$, we have $\partial_\infty M\subset L$.
To prove that $L\subset\partial_\infty M$, take an arbitrary point $x\in L$ and let $V$ be a (cone) neighborhood of $x$. Then there exists a geodesic $\gamma\colon\R\to N$, a closed tubular neighborhood $K$ of $\gamma(\R)$, with $K\subset V$, and points at infinity 
$\gamma(-\infty)\in U\cap V,\ \gamma(\infty)\in U'\cap V$. Hence for all sufficiently large $i$, $K$ intersects both $\partial B(o,R_i)\cap\cone_o U$ and $\partial B(o,R_i)\cap\cone_o U'$, and consequently $[\Gamma_i]$ can not be a boundary of a rectifiable $(n-1)$-current $T$ with $\spt T\subset \overline{B}(o,R_i)\setminus K$. In particular, $M_i$ intersects with $V$ for all sufficiently large $i$, and since $V$ is an arbitrary neighborhood of $x$, we conclude that 
$x\in\partial_\infty M$. Hence $L\subset\partial_\infty M$.
\end{proof}

\begin{remark}
\label{rmkregu}
Notice that, by regularity theory, a minimizer $M$ of the functional \eqref{massari} is a
$C^2$-smooth $(n-1)$-dimensional manifold up to a closed singular set of Hausdorff dimension at most $n-8$; see e.g. \cite[Teorema 5.1, 5.2]{massari}. In particular, in dimensions $n\leq 7$, a minimizer M is $C^2$-smooth.
\end{remark}

The next result is a generalization of \cite[Theorem 3]{Al-Ro}.
\begin{theorem}\label{main2b}
Let $N^n$ be a Cartan-Hadamard manifold satisfying the SC condition. Suppose also that sectional curvatures of $N$ have a negative upper bound
\[
K_N\le-\alpha^2<0,\quad\alpha>0.
\]
Let $L\subset\partial_\infty N$ be a relatively closed subset such that 
$L=\bd U=\bd U'$ for some disjoint relatively open subsets $U,U'\subset\partial_\infty N$ with $U\cup U'=\partial_\infty N\setminus L$.
If $H\in ]0,H_0]$ is a constant, there exist two disjoint open sets $Q_j\subset N, j=1,2$, of locally finite perimeter whose boundaries $M_j$ have constant (generalized) mean curvature $H$ and $\partial_\infty M_j=L$. 
\end{theorem}
\begin{proof}
Let $Q_1\subset N$, with $\partial_\infty Q_1\subset U\cup L$, be an open set of locally finite perimeter obtained in (the proof of) Theorem~\ref{main1}. Thus its boundary $M_1$ has constant (generalized) mean curvature towards $N\setminus Q_1$. Repeating the construction 
 with  $\Sigma_i$ replaced by $\Sigma_i^\prime=\partial B(o,R_i)\setminus (\Sigma_i\cup\Gamma_i)$ gives another open set $Q_2$ of locally finite perimeter whose boundary $M_2$ has constant mean curvature $H$ towards $N\setminus Q_2$. Furthermore, $\partial_\infty Q_2\subset U^\prime\cup L$ and $\partial_\infty M_j=L$.
 In the special case $H\equiv 0$, Theorem~\ref{main1} yields the existence of an open set $W\subset N$ of locally finite perimeter such that 
$S=\partial[W]$ is a (mass) minimizing locally rectifiable 
$(n-1)$-current (with integer multiplicity) and 
$\partial_\infty\spt S=L$; see also \cite[Theorem 1.6]{CHR3}. Then $Q_1\subset W$ and $Q_2\subset N\setminus\overline{W}$ since otherwise it would be possible to strictly decrease the value of the functional $\cF$ by replacing  
$M_1\cap(N\setminus\overline{W})$ with $\spt S\cap Q_1$, and similarly for $Q_2, M_2$. Hence $Q_1\cap Q_2=\emptyset$.
\end{proof}
\begin{remark}
Although $Q_1\cap Q_2=\emptyset$, the above reasoning does not imply that the boundaries $M_1$ and $M_2$ are (pointwise) disjoint. We believe that $M_1\cap M_2=\emptyset$ but, unfortunately, we are not able to prove it in full generality. However, using the regularity results from geometric measure theory it is possible to prove that the intersection $M_1\cap M_2$ is rather small and, in particular, an empty set in dimensions $n\le 7$;
see Remark \ref{rmkregu}.
\end{remark}
\begin{theorem}
Let $Q_i$ and $M_i,\ i=1,2$, be given by Theorem~\ref{main2}. Denote by $S_i$ the closed singular sets of Hausdorff dimension at most $n-8$ such that $M_i\setminus S_i,\ i=1,2$, are $C^2$-smooth submanifolds. Then $(M_1\setminus S_1)\cap (M_2\setminus S_2)=\emptyset$. In particular, in dimensions $n\le 7$, the sets $S_i$ are empty, and consequently $M_1$ and $M_2$ are disjoint.
\end{theorem}
\begin{proof}
Suppose on the contrary that $(M_1\setminus S_1)\cap (M_2\setminus S_2)\neq\emptyset$ and let 
$x\in (M_1\setminus S_1)\cap (M_2\setminus S_2)$. Then $M_i\setminus S_1$ has constant 
mean curvature $H$ at $x$ towards $N\setminus Q_i$. Now both $M_1$ and $M_2$ are subsets of $N\setminus Q_2$ intersecting tangentially at $x$. Since $M_1$ has mean curvature $-H$ at $x$ towards $N\setminus Q_2$ and $M_2$ has mean curvature $H$ at $x$ towards $N\setminus Q_2$, we get a contradiction since $H>0$.
\end{proof}


\begin{thebibliography}{10}

\bibitem{Al-Ro}
{\sc Alencar, H., and Rosenberg, H.}
\newblock Some remarks on the existence of hypersurfaces of constant mean
  curvature with a given boundary, or asymptotic boundary, in hyperbolic space.
\newblock {\em Bull. Sci. Math. 121}, 1 (1997), 61--69.

\bibitem{AmbFusPal}
{\sc Ambrosio, L., Fusco, N., and Pallara, D.}
\newblock {\em Functions of bounded variation and free discontinuity problems}.
\newblock Oxford Mathematical Monographs. The Clarendon Press, Oxford
  University Press, New York, 2000.

\bibitem{AndInv}
{\sc Anderson, M.~T.}
\newblock Complete minimal varieties in hyperbolic space.
\newblock {\em Invent. Math. 69}, 3 (1982), 477--494.

\bibitem{AndCMH}
{\sc Anderson, M.~T.}
\newblock Complete minimal hypersurfaces in hyperbolic {$n$}-manifolds.
\newblock {\em Comment. Math. Helv. 58}, 2 (1983), 264--290.

\bibitem{BL}
{\sc Bangert, V., and Lang, U.}
\newblock Trapping quasiminimizing submanifolds in spaces of negative
  curvature.
\newblock {\em Comment. Math. Helv. 71}, 1 (1996), 122--143.

\bibitem{CHR3}
{\sc Casteras, J.-B., Holopainen, I., and Ripoll, J.~B.}
\newblock Convexity at infinity in {C}artan-{H}adamard manifolds and
  applications to the asymptotic {D}irichlet and {P}lateau problems.
\newblock {\em Math. Z. 290}, 1-2 (2018), 221--250.

\bibitem{coskuGD}
{\sc Coskunuzer, B.}
\newblock Minimizing constant mean curvature hypersurfaces in hyperbolic space.
\newblock {\em Geom. Dedicata 118\/} (2006), 157--171.

\bibitem{cosku}
{\sc Coskunuzer, B.}
\newblock Asymptotic {P}lateau problem: a survey.
\newblock In {\em Proceedings of the {G}\"okova {G}eometry-{T}opology
  {C}onference 2013\/} (2014), G\"okova Geometry/Topology Conference (GGT),
  G\"okova, pp.~120--146.

\bibitem{coskupreprint}
{\sc Coskunuzer, B.}
\newblock Minimal surfaces with arbitrary topology in {$\mathbb{H}^2 \times
  \R$}.
\newblock {\em Preprint arXiv:1404.0214v2 [math.DG]\/} (2014).

\bibitem{coskuTAMS}
{\sc Coskunuzer, B.}
\newblock Embedded {$H$}-planes in hyperbolic 3-space.
\newblock {\em Trans. Amer. Math. Soc. 371}, 2 (2019), 1253--1269.

\bibitem{oliveirasoret}
{\sc de~Oliveira, G., and Soret, M.}
\newblock Complete minimal surfaces in hyperbolic space.
\newblock {\em Math. Ann. 311}, 3 (1998), 397--419.

\bibitem{EO}
{\sc Eberlein, P., and O'Neill, B.}
\newblock Visibility manifolds.
\newblock {\em Pacific J. Math. 46\/} (1973), 45--109.

\bibitem{federer}
{\sc Federer, H.}
\newblock {\em Geometric measure theory}.
\newblock Die Grundlehren der mathematischen Wissenschaften, Band 153.
  Springer-Verlag New York Inc., New York, 1969.

\bibitem{guanspruck}
{\sc Guan, B., and Spruck, J.}
\newblock Hypersurfaces of constant mean curvature in hyperbolic space with
  prescribed asymptotic boundary at infinity.
\newblock {\em Amer. J. Math. 122}, 5 (2000), 1039--1060.

\bibitem{gulliver}
{\sc Gulliver, II, R.~D.}
\newblock The {P}lateau problem for surfaces of prescribed mean curvature in a
  {R}iemannian manifold.
\newblock {\em J. Differential Geometry 8\/} (1973), 317--330.

\bibitem{Kasue}
{\sc Kasue, A.}
\newblock A {L}aplacian comparison theorem and function theoretic properties of
  a complete {R}iemannian manifold.
\newblock {\em Japan. J. Math. (N.S.) 8}, 2 (1982), 309--341.

\bibitem{KlMeRa}
{\sc Klaser, P., Menezes, A., and Ramos, A.}
\newblock {O}n the asymptotic {P}lateau problem for area minimizing surfaces in
  {$\mathbb{E} (-1,\tau )$}.
\newblock {\em Preprint arXiv:1905.03191 [math.DG]\/} (2019).

\bibitem{KlMa}
{\sc Kloeckner, B.~R., and Mazzeo, R.}
\newblock On the asymptotic behavior of minimal surfaces in {$\Bbb H^2\times
  \Bbb R$}.
\newblock {\em Indiana Univ. Math. J. 66}, 2 (2017), 631--658.

\bibitem{Lang1}
{\sc Lang, U.}
\newblock The existence of complete minimizing hypersurfaces in hyperbolic
  manifolds.
\newblock {\em Internat. J. Math. 6}, 1 (1995), 45--58.

\bibitem{Lang}
{\sc Lang, U.}
\newblock The asymptotic {P}lateau problem in {G}romov hyperbolic manifolds.
\newblock {\em Calc. Var. Partial Differential Equations 16}, 1 (2003), 31--46.

\bibitem{maggi}
{\sc Maggi, F.}
\newblock {\em Sets of finite perimeter and geometric variational problems},
  vol.~135 of {\em Cambridge Studies in Advanced Mathematics}.
\newblock Cambridge University Press, Cambridge, 2012.
\newblock An introduction to geometric measure theory.

\bibitem{martinwhite}
{\sc Mart\'{i}n, F., and White, B.}
\newblock Properly embedded, area-minimizing surfaces in hyperbolic 3-space.
\newblock {\em J. Differential Geom. 97}, 3 (2014), 515--544.

\bibitem{massari}
{\sc Massari, U.}
\newblock Esistenza e regolarit\`a delle ipersuperfice di curvatura media
  assegnata in {$R^{n}$}.
\newblock {\em Arch. Rational Mech. Anal. 55\/} (1974), 357--382.

\bibitem{morgan}
{\sc Morgan, F.}
\newblock {\em Geometric measure theory}, fourth~ed.
\newblock Elsevier/Academic Press, Amsterdam, 2009.
\newblock A beginner's guide.

\bibitem{nellispruck}
{\sc Nelli, B., and Spruck, J.}
\newblock On the existence and uniqueness of constant mean curvature
  hypersurfaces in hyperbolic space.
\newblock In {\em Geometric analysis and the calculus of variations}. Int.
  Press, Cambridge, MA, 1996, pp.~253--266.

\bibitem{RT}
{\sc Ripoll, J., and Telichevesky, M.}
\newblock Regularity at infinity of {H}adamard manifolds with respect to some
  elliptic operators and applications to asymptotic {D}irichlet problems.
\newblock {\em Trans. Amer. Math. Soc. 367}, 3 (2015), 1523--1541.

\bibitem{RT_BBMS}
{\sc Ripoll, J., and Telichevesky, M.}
\newblock On the asymptotic {P}lateau problem for {CMC} hypersurfaces in
  hyperbolic space.
\newblock {\em Bull. Braz. Math. Soc. (N.S.)\/} (2018).

\bibitem{RTo}
{\sc Ripoll, J., and Tomi, F.}
\newblock Complete minimal discs in {H}adamard manifolds.
\newblock {\em Adv. Calc. Var. 10}, 4 (2017), 315--330.

\bibitem{schmidt}
{\sc Schmidt, T.}
\newblock Strict interior approximation of sets of finite perimeter and
  functions of bounded variation.
\newblock {\em Proc. Amer. Math. Soc. 143}, 5 (2015), 2069--2084.

\bibitem{Tone}
{\sc Tonegawa, Y.}
\newblock Existence and regularity of constant mean curvature hypersurfaces in
  hyperbolic space.
\newblock {\em Math. Z. 221}, 4 (1996), 591--615.

\end{thebibliography}

\end{document}